\documentclass[11pt]{amsart}

\usepackage{amsmath}
\usepackage{amssymb}
\usepackage{amsthm}
\usepackage{graphicx} 
\usepackage{url}
\usepackage{pgf}
\usepackage{color}
\usepackage{hyperref}

\theoremstyle{plain}

\theoremstyle{plain}
\newtheorem{theorem}{Theorem}[section]

\newtheorem{lemma}[theorem]{Lemma}
\newtheorem{corollary}[theorem]{Corollary}

\theoremstyle{definition}
\newtheorem{definition}[theorem]{Definition}

\newcommand{\ds}{\displaystyle}

\newcommand{\bord}{\partial}

\renewcommand{\H}{\mathrm{H}}

\newcommand{\Hel}{\mathrm{Hel}}

\renewcommand{\le}{\leqslant}
\newcommand{\Lk}{\mathrm{Lk}}

\newcommand{\N}{\mathbb{N}}

\newcommand{\phiHopf}{\phi_{\mathrm{Hopf}}}

\newcommand{\R}{\mathbb{R}}

\newcommand{\Sph}{\mathbb{S}}

\newcommand{\Schw}{\mathcal{S}}
\newcommand{\Slk}{\mathrm{Slk}}

\newcommand{\T}{\mathbb{T}}

\newcommand{\XHopf}{X_{\mathrm{Hopf}}}

\newcommand{\Z}{\mathbb{Z}}
\newcommand{\zetaHopf}{\zeta_{\mathrm{Hopf}}}

\begin{document}

\title{Vector fields and genus in dimension 3}

\author{Pierre Dehornoy}
\address{Pierre Dehornoy, Univ. Grenoble Alpes, CNRS, Institut Fourier, F-38000 Grenoble, France}
\email{pierre.dehornoy@univ-grenoble-alpes.fr}
\urladdr{http://www-fourier.ujf-grenoble.fr/~dehornop/}
\thanks{PD is supported by the projects ANR-15-IDEX-02 and ANR-11LABX-0025-01}

\author{Ana Rechtman}
\address{Ana Rechtman, Institut de Recherche Math\'ematique Avanc\'ee,
Universit\'e de Strasbourg,
7 rue Ren\'e Descartes,
67084 Strasbourg, France}
\email{rechtman@math.unistra.fr}
\urladdr{https://irma.math.unistra.fr/~rechtman/}
\thanks{AR thanks the UMI LaSol and the support of the IdEx Unistra, Investments for the future program of the French Government.}


\begin{abstract}
	Given a vector field on a 3-dimensional rational homology sphere, 
	we give a formula for the Euler characteristic of its transverse surfaces, in terms of boundary data only. 
	This provides a formula for the genus of a transverse surface, and in particular, of a Birkhoff section.
	As an application, we show that for a right-handed flow with an ergodic invariant measure, 
	the genus is an asymptotic invariant of order 2 proportional to helicity.
\end{abstract}

\maketitle

\section{Introduction}

In this paper we study topological properties of non-singular vector fields on 3-dimensional homology spheres. 
An important problem in this field is to obtain measure-preserving homeomorphism or diffeomorphism invariants, meaning that the value is the same for flows that are conjugated by a homeomorphism or a diffeomorphism that preserve a given measure.
We provide a step towards the possible definition of such an invariant built from the genus of knots. 

Introduced by Woltjer, Moreau and Moffatt, {\it helicity} is the main known invariant~\cite{Woltjer, Moreau, Moffatt}. 
For a divergence free vector field $X$ on a closed Riemannian manifold $M$, it is defined by the formula~$\Hel(X) = \int X\cdot Y$, where $Y=\text{curl}^{-1}(X)$ is an arbitrary vector-potential of~$X$. 
Arnold and Vogel proved that, on homology spheres, helicity coincides with the average asymptotic linking number~\cite{Arnold, Vogel}. 
More precisely, let $\phi_X^t$ for~$t\in \R$ be the flow of $X$ and denote by~$k_X(p,t)$ the loop starting at the point $p$ that follows the orbit until~$\phi_X^t(p)$ and closes by an arbitrary segment of bounded length. 
The average asymptotic linking number is the double integral, with respect to an invariant measure, of $\lim_{t_1,t_2\to\infty}\frac{\Lk(k_X(p_1,t_1),k_X(p_2,t_2))}{t_1t_2}$, where $\Lk$ is the linking number between the two loops. 

In order to produce other asymptotic invariants, one is tempted to replace the linking number by another link or knot invariant. In this direction,
Freedman and He constructed the asymptotic crossing number~\cite{FH}, Gambaudo and Ghys constructed the asymptotic Ruelle invariant~\cite{GGRuelle} (see also Section~\ref{S:AsGenus}), while the authors of this paper constructed the trunkenness based on the trunk of a knot~\cite{DR}. 
These are three examples of invariants that are not proportional to helicity. 
On the other hand, Gambaudo and Ghys considered $\omega$-signatures of knots~\cite{GG}, Baader considered linear saddle invariants~\cite{Baader}, and Baader and March\'e considered Vassiliev's finite type invariants~\cite{BM}. 
All these constructions have the drawback that they do not yield any new invariant for ergodic vector fields, the obtained limits are all functions of the helicity.  
For an explanation for the ubiquity of helicity we refer to~\cite{Kudr,EPT}. 

The genus of a knot~$k$, denoted by~$g(k)$ and defined as the minimal genus of an orientable surface spanned by~$k$, is a fundamental invariant in knot-theory. 
An open problem in the context of vector fields is to prove that the limit of $ \frac{1}{t^n}g(k_X(p,t))$ exists for some $n$. 
Actually $n=2$ is the natural candidate as explained in~\cite[Question 5.4]{WBLN} and as a consequence of Theorem~\ref{T:AsGenus}.

In this paper we study the genus of surfaces whose boundary is composed by one or several periodic orbits of the flow and whose interior is transverse to the vector field. 
We refer to such surfaces as {\bf transverse surfaces}. 
Considering only such surfaces is a strong restriction, since  among the collection of  surfaces  with fixed boundary, those of minimal genus need not be transverse to the vector field. 
But a specific hypothesis on the flow will ensure that it is the case. 
Theorem~\ref{T:AsGenus} stands for {\bf right-handed vector fields}: these are vector fields on homology spheres all of whose invariant positive measures link positively~\cite{GhysLeftHanded}. This hypothesis implies that any collection of periodic orbits of the flow bounds a transverse surface that intersects all the orbits of the flow.

\begin{theorem}\label{T:AsGenus}
	Let $M$ be a 3-manifold that is a rational homology sphere, 
	$X$ a non-singular right-handed vector field on~$M$ and $\mu$ a $X$-invariant measure. 
	If $(\gamma_n)_{n\in\N}$ is a sequence of periodic orbits whose lengths $(t_n)_{n\in\N}$ tend to infinity 
	and such that $(\frac{1}{t_n}\gamma_n)_{n\in\N}$ tends to~$\mu$ in the weak-$*$ sense, 
	then the sequence $(\frac1{t_n^2}g(\gamma_n))_{n\in\N}$ tends to half the helicity of~$(X, \mu)$.  
\end{theorem}

Note that one can replace right-handedness by left-handedness, and only change the helicity by its absolute value. 
Right-, or left-, handedness is a strong restriction, but several important classes of vector fields have this property: for example the Lorenz vector field on~$\R^3$ is (in a certain sense) right-handed~\cite{GhysLeftHanded}, geodesic flows on positively curved surfaces are left-handed, as well as geodesic flows on hyperbolic triangular orbifolds~\cite{DLeftHanded}. 
On the other hand many flows are neither right- nor left-handed, as for example the Ghrist flow which contains all types of knots as periodic orbits~\cite{Ghrist}. 
For such general flows, we also expect the genus to have a quadratic asymptotic behaviour, but we expect the asymptotic value to be strictly larger than half the absolute value of the helicity. 

The proof of Theorem~\ref{T:AsGenus} mostly relies on an adaptation of results on global sections to flows to the case of surfaces with boundary, coupled with the classical fact in knot theory that fiber surfaces for knots are genus-minimizing. 
More precisely, a transverse surface is a {\bf Birkhoff section} if it intersects all the orbits of the flow. 
If the boundary is empty, one speaks of a {\bf global cross section}. 
When a flow admits a global cross section, up to changing the time-parameter (\emph{i.e.} multiplying the vector field by a striclty positive function), the dynamics of the flow is described by the dynamics of the first-return map on the global cross section. 
The description of all global cross sections to a vector field is given by Schwartzman-Fuller-Sullivan-Fried Theory, which is purely homological~\cite{Schwartzman, Fuller, Sullivan, FriedGeom}. 
Moreover, Thurston and Fried gave formulas for computing the genus of global cross sections~\cite{ThurstonNorm, FriedFLP}. 

Schwartzman-Fuller-Sullivan-Fried Theory may be extended to Birkhoff sections, but this has only been partially done~\cite{FriedGeom, GhysLeftHanded, Umberto}.
What we do here is to push a bit further Fried's and Ghys' ideas. 
We provide, in Corollary~\ref{Coro}, a formula for the genus of transverse surfaces with boundary, that depends only on data calculated along boundary components. 
The corollary is deduced from the following result that provides a formula for the Euler characteristic of these type of surfaces. 

\begin{theorem}\label{T:Chi}
	Assume that $M$ is a 3-dimensional rational homology sphere 
	and that $X$ is a non-singular vector field on~$M$. 
	Let $\{\gamma_i\}_{1\le i\le m}$ be a finite collection of periodic orbits of~$X$ 
	and $\{n_i\}_{1\le i\le m}$ a collection of integers. If 
	$S$ is a transverse surface to~$X$ with oriented boundary~$\cup n_i\gamma_i$, 
	then the Euler characteristic of~$S$ is given by
	\[
		\chi(S)
		=-\sum_{1\le i<j\le m} (n_i{+}n_j)\Lk(\gamma_i, \gamma_j) - \sum_{1\le i\le m} n_i\Slk^{\zeta_X}(\gamma_i),
	\]
	where $\zeta_X$ denotes any vector field everywhere transverse to~$X$ 
	and $\Slk^{\zeta_X}$ the self-linking given by the framing~$\zeta_X$ (see Definition~\ref{D:Slk}).
\end{theorem}

We can easily deduce the formula for the genus of the surface.

\begin{corollary}\label{Coro}
Assume that $M$ is a 3-dimensional rational homology sphere 
	and that $X$ is a non-singular vector field on~$M$. 
	Let $\{\gamma_i\}_{1\le i\le m}$ be a finite collection of periodic orbits of~$X$ 
	and $\{n_i\}_{1\le i\le m}$ a collection of integers. If 
	$S$ is a transverse surface to~$X$ with oriented boundary~$\cup n_i\gamma_i$, 
	then the genus of $S$ is given by
	\[
	g(S)
	= 1+\frac{1}{2}\left(\sum_{1\le i<j\le m} (n_i{+}n_j)\Lk(\gamma_i, \gamma_j) + \sum_{1\le i\le m} \left(n_i\Slk^{\zeta_X}(\gamma_i)-\gcd(n_i, \sum_{j\neq i}n_j \,\Lk(\gamma_i, \gamma_j))\right)\right).
	\]
\end{corollary}

It is likely that Theorem~\ref{T:Chi} may be adapted to an arbitrary 3-manifold~$M$. 
In this case, one would have to adapt the definition of linking number which is not anymore well-defined. 

\bigskip

Theorem~\ref{T:Chi} and Corollary~\ref{Coro} also have an independent interest. 
Given a flow in a 3-manifold, it is a natural question to look for Birkhoff sections of minimal genus or minimal Euler characteristics. 
For example Fried asked whether every transitive Anosov flow with orientable invariant foliations admits a genus-one Birkhoff section~\cite{FriedAnosov}. 
Similarly, Etnyre asked wether every contact structure can be defined by a contact form whose Reeb flow admits a genus-one Birkhoff section~\cite{Etnyre}. 
Corollary~\ref{Coro} was implemented by the first author for the geodesic flows on some hyperbolic orbifolds, which led to answer positively Fried's question in these cases~\cite{GenusOne}. 
It was also used by Dehornoy and Shannon to numerically check that suspensions of linear automorphisms of~$\T^2$ admit infinitely many genus-one Birkhoff sections~\cite{DShannon}. 

\bigskip

The paper is organized as follows. 
In Section~\ref{sec:proofmain} we give a proof of Theorem~\ref{T:Chi} and Corollary~\ref{Coro}. 
In Section~\ref{s:examples} we illustrate Theorem~\ref{T:Chi} with the example of the Hopf vector field.
Theorem~\ref{T:AsGenus} is proved in Section~\ref{s:asymptotic}.

\bigskip

\noindent{\bf Acknowledgments.} The authors thank Adrien Boulanger, \'Etienne Ghys and Christine Lescop for several discussions around the topic of this paper, and the referee for several suggestions that hopefully improve the readibility of the paper.


\section{The Euler characteristic of a transverse surface}\label{sec:proofmain}

The aim of this section is to give a proof of Theorem~\ref{T:Chi},
which is done in Section~\ref{S:QHS}. 
We first recall in~\ref{S:SFSF} some classical results on global cross sections and we explain in~\ref{S:Genus} how the genus of such a section can be obtained. 
Section~\ref{S:Linking} introduces linking and self-linking, 
and~\ref{S:Boundary} contains the key-lemma for the proof of Theorem~\ref{T:Chi}.


\subsection{Schwartzman-Fuller-Sullivan-Fried Theory}\label{S:SFSF}

We recall classical results concerning global cross sections to flows, but we state them in the more general context of a 3-manifold~$M$ with toric boundary and a non-singular vector field~$X$ tangent to~$\partial M$. 
The original proofs extend verbatim to this case. 

First we recall the definition of {\bf asymptotic cycles}. 
The original one is in terms of almost-periodic orbits~\cite{Schwartzman}: 
for $p\in M$ and $t>0$, we denote by $k_X(p, t)$ the closed curve obtained by connecting the arc of orbit~$\phi_X^{[0,t]}(p)$ with a segment of bounded length (recall that $M$ is compact). 
Assume now that $\mu$ is an ergodic invariant positive measure and that $p$ is a quasi-regular point for~$\mu$. 
The asymptotic cycle~$a_\mu$ determined by~$\mu$ is the weak-$*$ limit~$(\frac{1}{t_n}[k_X(p, t_n)])_{n\in\N}$, with $t_n\to\infty$. 
It is independent of~$p$ and~$t_n$. 
The set~$\Schw_X$ of all asymptotic cycles is the convex hull of those asymptotic cycles associated to ergodic invariant positive measures. 
It is a convex cone in~$\H_1(M; \R)$.

Alternatively, for $\mu$ a $X$-invariant positive measure, one can consider the 1-current $c_\mu:\Omega^1(M)\to\R$ which maps a 1-form $f$ to~$\int_M f(X(p))\,d\mu(p)$. The invariance of~$\mu$ implies that~$c_\mu$ is closed ({\it i.e.}, it vanishes on closed forms), hence it determines a 1-cycle $[c_\mu]\in\H_1(M; \R)$ in the sense of De Rham. 
The two notions actually coincide: when $\mu$ is an ergodic measure, $a_\mu$ and $[c_\mu]$ are equal under the identification of singular and current homologies~\cite{Sullivan}. 

Schwartzman's criterion~\cite{Schwartzman} is the following.  

\begin{theorem}[Schwartzman]
	A class~$\sigma\in \H^1(M; \Z)$ is dual to a global cross section 
	if, and only if, for every asymptotic cycle~$c\in\Schw_X$ one has $c(\sigma) >0$.
\end{theorem}


\subsection{Genus of global cross sections}\label{S:Genus}

In the context of the previous part, a standard argument shows that if two global cross sections to a vector field are homologous, then they are isotopic along the flow~\cite{ThurstonNorm}. 
Actually it shows more: a global cross section minimises the genus in its homology class. 
So one may wonder how to compute this genus. 
Thurston and Fried give a satisfying answer~\cite{FriedFLP}. 
Denote by~$X^\perp$ the normal bundle to~$X$ (it is the 2-dimensional bundle~$TM/\R X$), and by $e(X^\perp)\in\H^2(M, \partial M; \Z)$ its Euler class.

\begin{theorem}[Fried-Thurston]\label{thm: ft}
	Assume that $S$ is a surface transverse to~$X$. 
	Then one has $\chi(S) = e(X^\perp)([S])$. 
\end{theorem}

The argument is short: since $S$ is transverse to~$X$, the restricted bundle~$X^\perp|_S$ is isomorphic to the tangent bundle~$TS$. 
In particular one has $e(X^\perp)([S]) = e(TS)([S])=\chi(S)$. 

Said differently, if $\zeta$ is a vector field in generic position with respect to~$X$, the set $L_{\zeta, X}$ where~$\zeta$ is tangent to~$X$ is a 1-manifold. 
In order to consider its homology class, one has to orient~$L_{\zeta, X}$ and to equip it with multiplicities. 
Here is how one can do it~(see Figure~\ref{F:OrientationEuler}): 
one takes a small disc~$D$ positively transverse to~$X$ and transverse to~$L_{\zeta, X}$. 
The projection of~$\zeta$ on~$D$ along~$X$ defines a vector field~$\zeta_D$ on~$D$ with a singularity at the center. 
The index of the singularity may be positive or negative (it cannot be 0 for in this case $X$ and $\zeta$ would not be in generic position) and thus the product of its sign with the orientation of~$D$ induces a new orientation on~$D$. 
Since $D$ is transverse to~$L_{\zeta, X}$, this new orientation induces an orientation of~$L_{\zeta, X}$. 
The multiplicity then comes from the absolute value of the index of the singularity. 
Observe that the multiplicity is locally constant by continuity and hence constant on each connected component of $L_{\zeta,X}$. 
The point is that these choices are independent of~$D$. 
If one perturbes $D$ by keeping it transverse to~$X$ and~$L_{\zeta, X}$, by continuity and discreteness, the multiplicity and orientation do not change. 
If one perturbes~$D$ by keeping it transverse to~$X$ but change the relative position to respect to~$TL_{\zeta,X}$, the induced orientation on~$L_{\zeta, X}$ changes, but the index is also changed by its opposite. 
Hence the product is constant.

\begin{figure}[h!]
	\centering
	\includegraphics*[width=.6\textwidth]{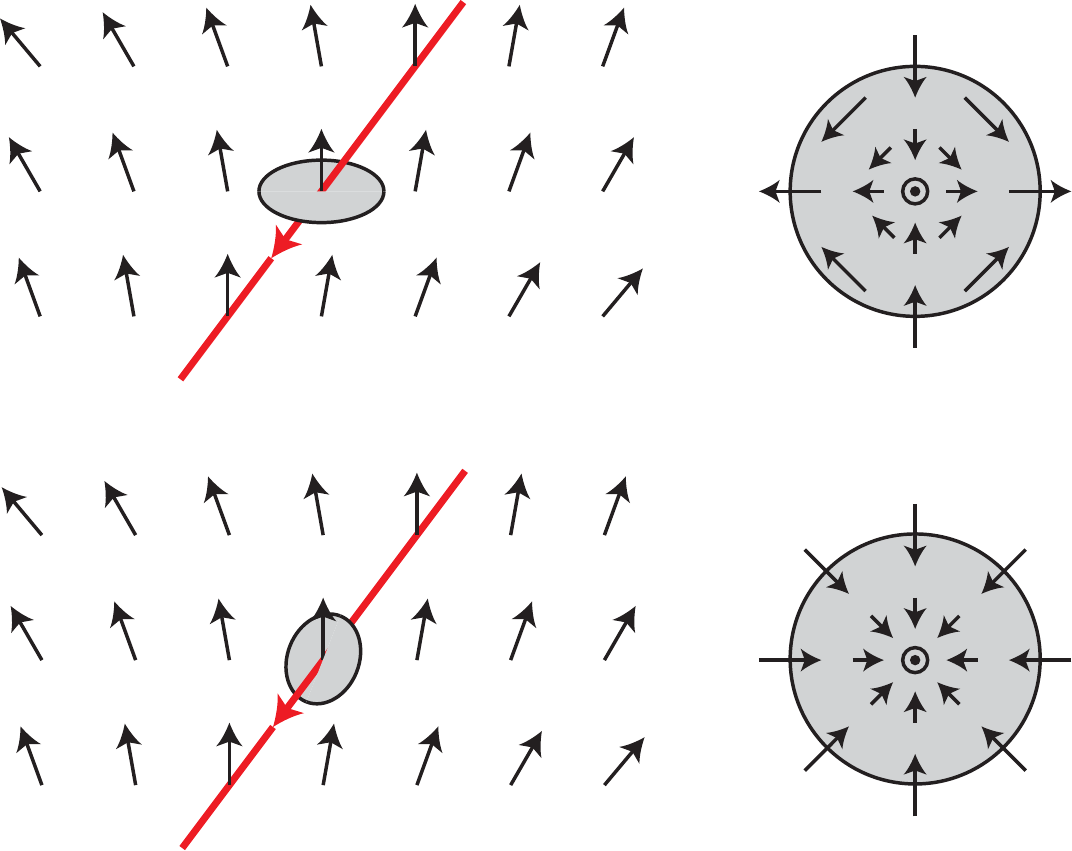}
   	\caption{\small Orientation of the dual~$L_{\zeta, X}$ of the Euler class~$e(X^\perp)$. 
	On this picture the vector field $X$ is locally vertical. 
	A vector field~$\zeta$ in general position with respect to~$X$ is shown. 
	It is tangent to $X$ (\emph{i.e.}, vertical) on a dimension 1-submanifold~$L_{\zeta, X}:=\{X\parallel \zeta\}$, in red. 
	We consider an arbitrary disc~$D$ transverse to both~$X$ and $L_{\zeta,X}$. 
	Projecting $\zeta$ on~$D$ along~$X$ we get a vector field~$\zeta_D$ with one singularity at the center. 
	The index of this vector field gives a multiplicity to the disc, 
	and together with the orientation of~$D$ an orientation and a multiplicity to~$L_{\zeta, X}$. 
	One checks that changing~$D$ may change the index by its opposite (bottom), 
	but also the orientation, so that their product is unchanged.}
  	\label{F:OrientationEuler}
\end{figure}

The class~$[L_{\zeta, X}]\in\H_1(M;\Z)$ is then Poincar\'e dual to~$e(X^\perp)$, see~\cite[Prop. 12.8]{BottTu}. 
Given a surface~$S$ transverse to~$X$, projecting~$\zeta$ on~$S$ along~$X$ yields a vector field~$\zeta_S$ on~$S$, which vanishes exactly when~$S$ intersects~$L_{\zeta, X}$. 
The Euler characteristic of~$S$ can be computed with the Poincar\'e-Hopf formula for~$\zeta_S$. 
The crucial point~\cite{ThurstonNorm, FriedFLP} is that, thanks to the orientation and multiplicity of~$L_{\zeta, X}$, each intersection point contributes with the right sign to the sum.  



\subsection{Linking and self-linking on homology spheres}\label{S:Linking}

We now assume that $M$ is a rational homology sphere. 
Given two links $L_1, L_2$, their {\bf linking number}~$\Lk(L_1, L_2)$ is defined as~$\langle L_1, S_2\rangle$, where $S_2$ is a rational 2-chain bounded by~$L_2$. 
The existence of~$S_2$ is guaranteed by the vanishing of~$[L_2]$, whereas the vanishing of~$[L_1]$ implies that the linking number is independent of the choice of~$S_2$. 
Linking number is symmetric, although this is not obvious from the definition we just gave.

A {\bf framing}  $f$ of a link~$L$ is a section of its unit normal bundle. 
It induces an isotopy class~$L^f$ in~$M\setminus L$ obtained by pushing~$L$ off itself in the direction of~$f$. 

\begin{definition}\label{D:Slk}
	Given a link~$L$ and a framing~$f$, the {\bf self-linking}~$\Slk^f(L)$ is defined 
	as the linking number of~$L$ and $L^f$.
\end{definition}

When $M$ is an integral homology sphere, every individual curve (that is, the curve with multiplicity one) has a preferred framing corresponding to a self-linking number zero. 
By definition, this {\bf zero-framing} is induced by any surface bounded by the considered curve.  
Therefore one can see the self-linking number  with respect to a framing $f$ as the algebraic intersection number between the framing $f$ and a surface bounded by the considered curve.

When $M$ is a rational non-integral homology sphere, there is not always a preferred framing in the above sense. 
One then has to consider {\bf rational framings} which are multi-sections of the unit normal bundle. 
More precisely, the unit normal bundle to a curve $\gamma$ is a torus~$\gamma\times\Sph^1$. 
A rational framing is a homotopy class of a closed curve on that torus. 
Given such a rational framing~$f$ of a link~$L$, we can extend the definition of self-linking to this context.  
Assume that $f$ winds $k_f$ times in the longitudinal direction along~$L$, we define~$L^f$ as the link obtained by pushing the link $L$ traveled~$k_f$ times off itself in the direction of~$f$. 
Then $\Slk^f(L)$ is defined as~$\frac1{k_f}\Lk(L, L^f)$. 
With this extended definition, there is always one (rational) zero-framing.
%
%


\subsection{Boundary slope}\label{S:Boundary}

We continue with the assumption that $M$ is a rational homology sphere. Fix a link  $L=K_1\cup\dots\cup K_m$ in~$M$.  
The boundary operator~$\bord$ realises an isomorphism~$\H_2(M, L; \R)\simeq \H_1(L; \R)$, as can be seen by writing the long exact sequence.
Therefore a class in $\H_2(M, L; \R)$ is determined by its boundary class. 
For~$\sigma$ in~$\H_2(M, L; \R)$, denote by~$n_i(\sigma)$ its {\bf longitudinal boundary coordinates}, that is, the real numbers such that $\bord\sigma = \sum n_i(\sigma)[K_i]$. 

In this context, we denote by~$M_L$ the {\bf normal compactification} of~$M\setminus L$: the manifold obtained from~$M$ by replacing every point of~$L$ by the circle of those half-planes bounded by~$TL$. 
The boundary~$\bord M_L$ is then isomorphic to~$L\times\Sph^1$: it is
a disjoint union of tori.
The manifold $M_L$ is actually isomorphic to~$M\setminus\nu(L)$, where $\nu(L)$ is an open tubular neighbourhood of~$L$.

By the excision theorem, there is an isomorphism~$\H_2(M, L; \R)\simeq  \H_2(M_L, \bord M_L; \R)$, so that we can also see the class~$\sigma$ as an element of~$\H_2(M_L, \bord M_L; \R)$. 
There, its boundary is an element of~$\H_1(\bord M_L; \R)$, whose dimension is higher than the dimension of~$\H_1(L; \R)$. 
For distinction we then write $\bord^\bullet:\H_2(M, L; \R)\to\H_1(L; \R)$ and $\bord^\circ:\H_2(M_L, \bord M_L; \R)\to\H_1(\bord M_L; \R)$ for the two operators (see Figure~\ref{F:Boundary}). 
As we said the first one is an isomorphism, while the second one is only injective. 

\begin{figure}[h!]
	\centering
	\includegraphics{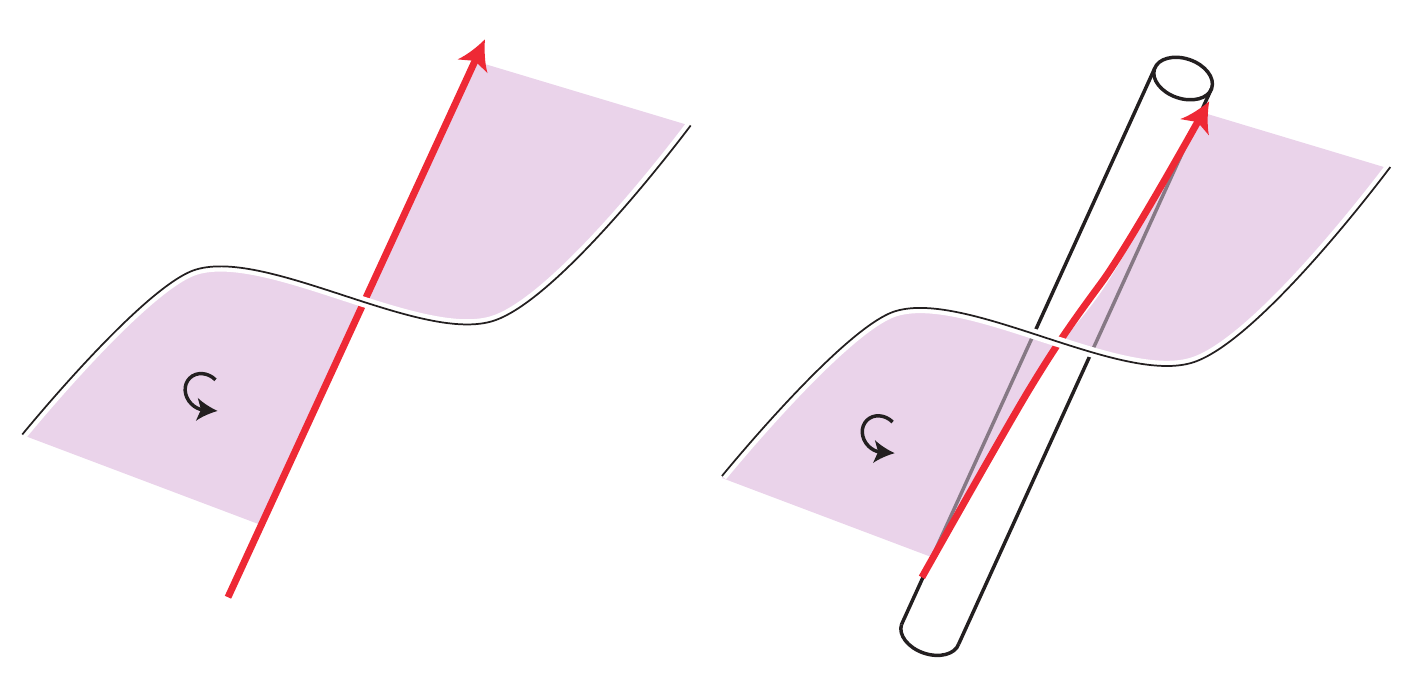}
	\caption{On the left, a link~$L$ (red) in a 3-manifold~$M$ and a surface $S$ (purple) representing 
	a class~$\sigma$ with $\bord^\bullet \sigma = L$. 
	On the right, the corresponding manifold~$M_L$ with boundary~$L\times\Sph^1$ 
	and the corresponding surface~$S$ whose boundary~$\bord^\circ S$ sits in~$L\times\Sph^1$. 
	The additional information given by the meridian coordinate of~$\bord^\circ S$ tells 
	how many times~$S$ wraps around~$L$.}
	\label{F:Boundary}
\end{figure}

Let $S$ be a surface representing the class $\sigma$ above. In order to understand the image of~$\bord^\circ$, we have to understand the framing induced by~$S$ along every boundary component (that is, the slope of~$S\cap(K_i\times\Sph^1)$ for every component $K_i$ of $L$). 

\begin{lemma}\label{L:Framing}
	If $M$ is a rational homology sphere and $S$ is a surface with~$\bord^\bullet S=\sum_i n_i K_i$, 
	then the coordinates of $\bord^\circ S$ along~$K_i$ in the (meridian, 0-longitude)-basis 
	are~$\displaystyle{\bigg(-\sum_{j\neq i} n_j \,\Lk(K_i, K_j),n_i\bigg)}$. 
\end{lemma}

\begin{proof}
Since $M$ is a homology sphere, all oriented surfaces with the same boundary induce the same framing on the boundary. 
An oriented surface realizing~$[S]$ is obtained by desingularising the union $\cup_i n_i(\sigma) S_i$ where $S_i$ is an oriented surface in~$M$ with boundary~$K_i$. 
Observe that the intersection of $S_i$ and $S_j$, for $i\neq j$, can be made either empty or transverse, and when non-empty it can be made of segments with one endpoint in $K_i$ and the other in $K_j$. 

If the intersection is non-empty, at each connected component of the intersection we have two choices of desingularisation, but only one that respects orientations. 
The desingularisation near the endpoints of each segment in the intersection is obtained by removing one meridian to~$K_i$ everytime $S_j$ intersects~$K_i$. 
This number is equal to $\Lk(K_i,K_j)$ and hence the total meridian contribution of all surfaces on~$K_i$ is~$-\sum_{j\neq i} n_j \Lk(K_i, K_j)$. 
On the other hand the longitudinal coordinate is unchanged in this desingularisation process, thus it is~$n_i$. The conclusion follows immediately.
\end{proof}


\subsection{The Euler class of $X_\Gamma^\perp$}\label{S:QHS}

Assume that $M$ is a rational homology sphere with empty boundary. 
We are  given a non-singular vector field~$X$ on~$M$, 
a finite collection~$\Gamma=\gamma_1\cup\dots\cup \gamma_m$ of periodic orbits of~$X$ 
and multiplicities $n_1, \dots, n_m$ which are integers. 
In this context the existence of a Birkhoff section for~$X$ bounded
by~$\ds{\cup_{i=1}^m} n_i\gamma_i$ is the same as the existence of a
cross section~$(S, \bord S)$ for the extension of $X$ to the
manifold~${M_\Gamma}$ such that the longitudinal coordinates of~$\bord^\circ S$ are~$(n_1, \dots, n_m)$.

Denote by~$X_\Gamma$ the extension of~$X$ to~$M_\Gamma$. 
In order to understand the  topology of the cross section,  
one wonders whether the class $e(X_\Gamma^\perp)$ may be easily represented. 

Since the Euler class of~$X^\perp$ vanishes, there exists non-singular vector fields on~$M$ everywhere transverse to~$X$. 
Since $M$ is a homology sphere, two such vector fields are homotopic through vector fields that are everywhere transverse to~$X$. 
Indeed the first vector field gives an origin to the normal sphere bundle, so that the second vector field yields a function on the circle. 
Since $M$ is a homology sphere, this function is homotopic to a constant function.

Denote by~$\zeta_X$  a vector field transverse to $X$. 
We use $\zeta_X$ to realise  the Euler class~$e(X_\Gamma^\perp)$. As in Fried and Thurston theorem (see Theorem~\ref{thm: ft}), the Euler class of a vector bundle is represented by the intersection with a generic section (see~\cite[Prop. 12.8]{BottTu}). 
Here one has to take into account the fact that $M_\Gamma$ has boundary. 

\begin{proof}{[Proof of Theorem~\ref{T:Chi}]}
The vector field $\zeta_X$ is everywhere transverse to~$X$ on~$M$ but not tangent to~$\bord M_\Gamma$. 
In order to make it tangent to the boundary of the manifold, we have to ``rotate''~$\zeta_X$ towards~$X$ around each component~$\gamma_i$. This can be achieved by combining, near $\gamma_i$, the two vector fields $\zeta_X$ and $X_\Gamma$ (see Figure~\ref{F:Zeta}).

We obtain a vector field $\zeta_{X, \Gamma}$ which is transverse to~$X_\Gamma$ on~$M_\Gamma$, except at the boundary components, where it is tangent to~$X_\Gamma$ along two curves that correspond to the framings given by~$\zeta_{X}$ and~$-\zeta_X$. 
In particular the set $L_{\zeta_X, \Gamma}:=\{p\in M_\Gamma\,|\,X_\Gamma(p) \parallel \zeta_{X, \Gamma}(p)\}$ is a collection of two longitudes~$\gamma_i^{\mathrm{in}}, \gamma_i^{\mathrm{out}}$ for every boundary component~$\gamma_i\times\Sph^1$ of $\bord M_\Gamma$.

\begin{figure}[h!]
	\centering
	\includegraphics{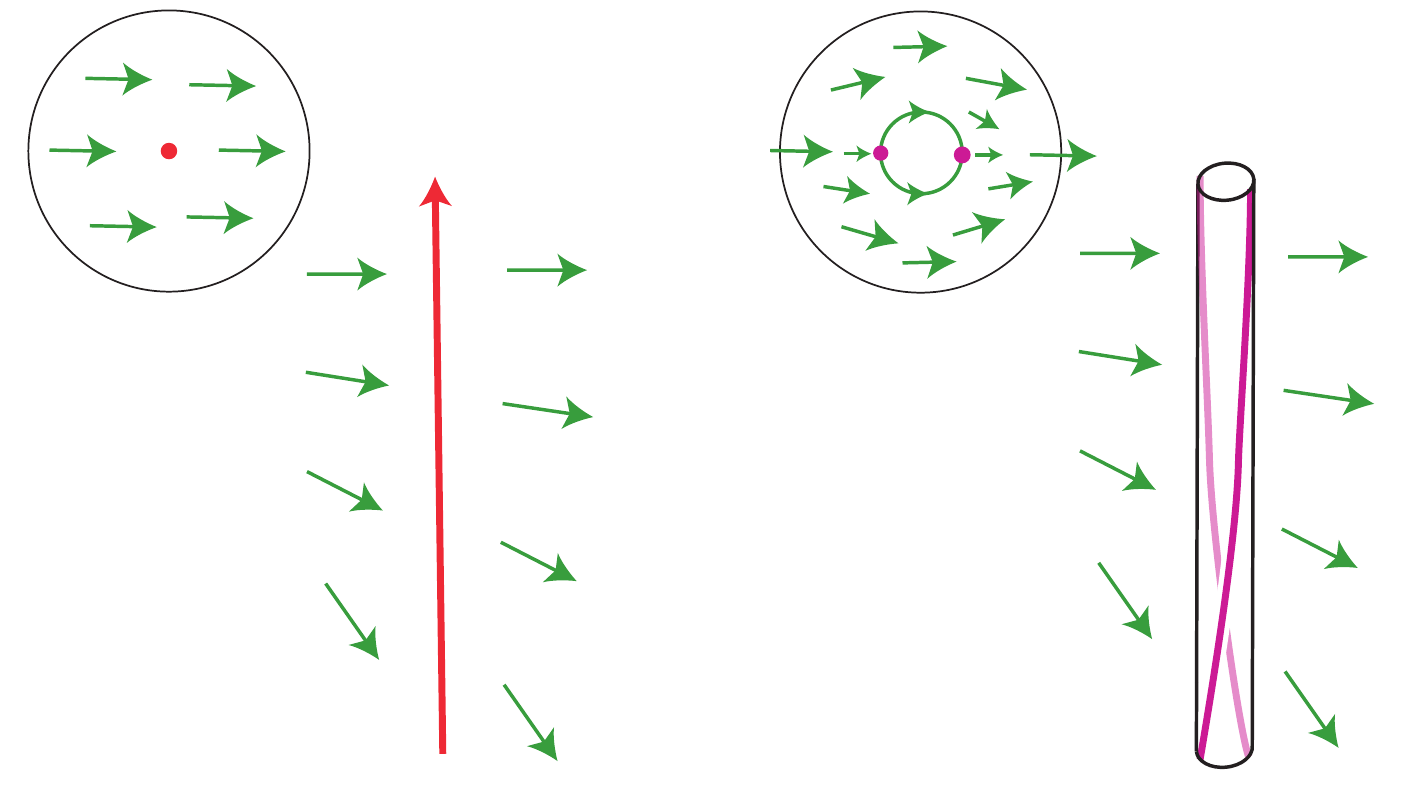}
	\caption{On the left, the vector field~$\zeta_X$ (green) on~$M$. 
	Since it is transverse to~$X$, is it transverse to the link~$\Gamma$ (red) made of periodic orbits of~$X$. 
	Seen from above, $\Gamma$ is a point and $\zeta_X$ is a non-vanishing vector field. 
	On the right, the modification of~$\zeta_X$ into~$\zeta_{X, \Gamma}$. 
	Seen from above, one has to slow down~$\zeta_X$ so that it has (transversal) speed~0 on~$\Gamma$. 
	The set~$L_{\zeta_X, \Gamma}$ (pink) then consists of two longitudes per component of~$\Gamma$.
	}
	\label{F:Zeta}
\end{figure}

Orienting $L_{\zeta_X, \Gamma}$ so that its class is dual to~$e(X_\Gamma^\perp)$ can be done as in Section~\ref{S:Genus}. 
Considering a component~$\gamma_i$ of~$\Gamma$, one has to take a small disc~$D$ in~$M$ transverse to~$\gamma_i$. 
The corresponding annulus $D_\Gamma$ in $M_\Gamma$ is automatically transverse to~$L_{\zeta_X, \Gamma}$. 
The projection of~$\zeta_{X,\Gamma}$ on~$D_\Gamma$ exhibits two singularities of index~$-\frac 12$ (see the right-hand disc in  Figure~\ref{F:Zeta}). 
Therefore, if~$\gamma_i^{\mathrm{in}}$ and $\gamma_i^{\mathrm{out}}$ are oriented in the direction opposite to~$X$, they both have multiplicty~$\frac 12$. 

Let $S$ be the surface transverse to $X$ in the statement of Theorem~\ref{T:Chi} and  $\sigma\in\H_2(M_\Gamma, \bord M_\Gamma; \Z)$ be its class, then $\chi(S)=e(X_{\Gamma}^\perp)(\sigma) = \langle L_{\zeta_X, \Gamma}, \sigma\rangle$. 
This intersection equals
\[\sum_{i=1}^m \langle \gamma_i^\mathrm{in}+\gamma_i^\mathrm{out}, (\bord^\circ \sigma)_i \rangle, \]
where $(\bord^\circ \sigma)_i$ denotes the part of $\bord^\circ \sigma$ on the component~$\gamma_i\times \Sph^1$ of $\bord M_\Gamma$.

Now the algebraic intersection of two curves on a 2-torus is the determinant of their coordinates in homology.
Since $ \gamma_i^\mathrm{in}$ and $\gamma_i^\mathrm{out}$ both have coordinates $(-\Slk^{\zeta_X}(\gamma_i),-1)$ in the (meridian, 0-longitude)-basis of~$\gamma_i\times\Sph^1$, and since every intersection point contributes to~$\frac 12$ to the intersection, thanks to Lemma~\ref{L:Framing}, we have
 \[\langle \gamma_i^\mathrm{in}+\gamma_i^\mathrm{out}, (\bord^\circ \sigma)_i \rangle
 = -n_i\Slk^{\zeta_X}(\gamma_i) - \sum_{j\neq i} n_j \,\Lk(K_i, K_j).\]

Summing over all boundary components, we get the desired formula.
\end{proof}

\begin{proof}{[Proof of Corollary~\ref{Coro}]}
From the surface $S$ bounded by~$\cup n_i\gamma_i$ described above, we only have to cap all boundary components with discs for obtaining a closed surface. 
However counting how many discs we need is a bit subtle: along each orbit~$\gamma_i$, the surface winds $n_i$ times longitudinally and, thanks to Lemma~\ref{L:Framing}, $-\sum_{j\neq i}n_j \,\Lk(\gamma_i, \gamma_j)$ meridionally. 
Therefore the number of boundary components along~$\gamma_i$ of the abstract surface~$S$ is the $\gcd$ of these numbers. 
We get
\begin{eqnarray*}
g(S)&=&1-\frac{\chi(S)+\sum_i \gcd(n_i, \sum_{j\neq i}n_j \,\Lk(\gamma_i, \gamma_j))}{2}\\
&=& 1+\frac{1}{2}\left(\sum_{1\le i<j\le m} (n_i{+}n_j)\Lk(\gamma_i, \gamma_j) + \sum_{1\le i\le m} \left(n_i\Slk^{\zeta_X}(\gamma_i)-\gcd(n_i, \sum_{j\neq i}n_j \,\Lk(\gamma_i, \gamma_j))\right)\right).
\end{eqnarray*}
\end{proof}


\section{Examples: Birkhoff sections for the Hopf vector field}\label{s:examples}

Let $\XHopf$ denote the Hopf vector field on~$\Sph^3$ and $\phiHopf$ denote the associated flow. 
Every orbit~$\gamma$ of~$\phiHopf$ is periodic and bounds a disc~$D_\gamma$ transverse to~$\XHopf$. 
Since any two orbits have linking number~$+1$, the disc~$D_\gamma$ is a Birkhoff section for~$\phiHopf$.
More generally, if $\cup n_i\gamma_i$ and $\cup n_i'\gamma_i'$ are two collections of disjoint periodic orbits with multiplicities, their linking number is given by~$(\sum n_i)(\sum n_i')$. 

On the other hand, a vector field transverse to~$\XHopf$ can be easily found by taking another Hopf fibration $\zetaHopf$ {orthogonal} to~$\XHopf$. 
One checks that for every periodic orbit~$\gamma$ of~$\phiHopf$, one has $\Slk^{\zetaHopf}(\gamma)=-1$. 

For a Birkhoff section that is a disc~$D_\gamma$ with one boundary component~$\gamma$, Theorem~\ref{T:Chi} then yields
\[\chi(D_\gamma) = -(-1)=+1,\]
as expected. 

Consider now a collection $\gamma_1, \dots, \gamma_m$ of $m$ periodic orbits. 
The link~$\Gamma:=\sum_i\gamma_i$ links positively with any positive invariant measure. 
Hence it bounds a Birkhoff section, denoted by~$S_\Gamma$. 
This section can be obtained from the union~$D_{\gamma_1} \cup\dots\cup D_{\gamma_m}$ by desingularising along the segments where two discs intersect. As discussed in the proof of Lemma~\ref{L:Framing}, for each segment there are two possible ways to resolve the intersection an obtain a (non-orientable) surface with the same boundary. 
But since here we want the resulting surface to be transverse to the vector field $\XHopf$, there is only one way to desingularise.
 
The Euler characteristic of the resulting surface can be computed by hand, but Theorem~\ref{T:Chi}  directly yields
\[\chi(S_\Gamma) = -m(m-1)-(-m)= -m(m-2).\]
Since $S_\Gamma$ has $m$ boundary components, we obtain 
\[g(S_\Gamma)=1-\frac{\chi(S_\Gamma)+m}{2}= 1+\frac{m(m-3)}{2}\]
which is the genus of a Hopf link with $m$ components. %
One can generalise a bit more by considering a collection~$\cup_in_i\gamma_i$, where the $n_i$ are integers. 
The condition for bounding a Birkhoff section then becomes $\sum_i n_i>0$,  since the linking number of $\cup_in_i\gamma_i$ with any other orbit of the flow has to be positive. 
Denoting by~$S_{\cup_in_i\gamma_i}$ such a Birkhoff section, 
Theorem~\ref{T:Chi} yields 
\[	
	\chi(S_{\cup_in_i\gamma_i})=-\sum_{1\le i<j\le m} (n_i+n_j) + \sum_{1\le i\le m} n_i=  \sum_{1\le i\le m}(1-m)n_i+ \sum_{1\le i\le m}n_i= (2-m)\sum_{1\le i\le m} n_i.
\]


\section{Genus and the Ruelle invariant}\label{s:asymptotic}

In this section we prove  Theorem~\ref{T:AsGenus}. 
As discussed above, when $M$ is a homology sphere and $X$ a non-singular vector field on~$M$, 
for every periodic orbit~$\gamma$ of $X$, we have presented two preferred framings, 
namely the zero-framing determined by a spanning surface, and the framing given by a vector field~$\zeta_X$ everywhere transverse to~$X$. 
By definition, the difference of these two framings along $\gamma$ is~$\pm\Slk^{\zeta_X}(\gamma)$. 

If $\gamma$ is the unique boundary component of a Birkhoff section, Corollary~\ref{Coro} says that the genus of this Birkhoff section is~$1+(\Slk^{\zeta_X}(\gamma)-1)/2$. 
One wonders whether this quantity has an asymptotic behaviour when~$\gamma$ tends to fill~$M$. 
Two related quantities are known to have one, and we present them now. 
Both rely on a third framing on~$\gamma$, given by the differential of the flow.

\subsection{Three framings on~$\gamma\times\Sph^1$}
Recall that $M_\gamma$ is the 3-manifold with boundary obtained from~$M$ by replacing every point of~$\gamma$ by its sphere normal bundle~$S(TM/\R X)$ which is topologically a circle. 
If $X$ is at least~$C^1$, we can then extend~$X$ to~$\gamma\times\Sph^1$ using the differential of the flow, and obtain a non-singular vector field~$X_\gamma$ on~$M_\gamma$. 
Now we have three framings on~$\gamma\times\Sph^1$, two integral ones  (the zero-framing and the one induced by~$\zeta_X$) and one real (induced by~$DX$). 

The restriction of~$X_\gamma$ to~$\bord M_\gamma$ is a vector field on a torus whose first coordinate (in the $X$-direction) can be made constant. 
Hence it has a well-defined translation number: 
the Ruelle invariant~$R^X(\gamma)$  is defined as the translation number of~$X_\gamma|_{\gamma\times\Sph^1}$ with respect to the framing~$\zeta_X$~\cite{Ruelle,GGRuelle}.
On the other hand the rotation number of $X_\gamma|_{\gamma\times\Sph^1}$ with respect to the zero-framing is given by~$\Slk^{DX}(\gamma)$.
Both these numbers are real (and not necessarily integers).

Since the quantities $R^X(\gamma), \Slk^{\zeta_X}(\gamma)$ and $\Slk^{DX}(\gamma)$ denote the respective difference between the three possible pairs of framings, we have
\begin{equation}
\label{Eq:Framings} \Slk^{\zeta_X}(\gamma) = \Slk^{DX}(\gamma) - R^X(\gamma),
\end{equation}
where only the term $\Slk^{\zeta_X}(\gamma)$ is always an integer.

The Ruelle invariant may be extended to any $X$-invariant measure using long arcs of orbits, but we do not need this here (see~\cite{GGRuelle}).


\subsection{Asymptotic genus for right-handed vector fields}\label{S:AsGenus}

Assume that $X$ is now a right-handed vector field on a rational homology sphere~$M$, meaning that all $X$-invariant positive measures have positive linking number~\cite{GhysLeftHanded}. 
In this context Ghys proved that every periodic orbit bounds a Birkhoff section. 
It is also known that such a section is genus-minimizing (this is a folklore result among 3-dimensional topologists, a possible reference is~\cite{ThurstonNorm} although the statement is older). 
Therefore the genus of a periodic orbit~$\gamma$ is given by~$1+(\Slk^{\zeta_X}(\gamma)-1)/2$. 

\begin{proof}{[Proof of Theorem~\ref{T:AsGenus}]}
Arnold and Vogel proved that if $(\gamma_n)$ is a sequence a periodic orbits that tend to an invariant volume~$\mu$ in the weak-$*$ sense, then writing~$t_n$ for the period of~$\gamma_n$, the sequence~$\frac 1{t_n^2} \Slk^{DX}(\gamma_n)$ tends to the helicity~$\Hel(X, \mu)$~\cite{Arnold, Vogel}. 
Similarly, Gambaudo and Ghys proved that the sequence~$\frac 1{t_n} R^X(\gamma_n)$ tends to the Ruelle invariant~$R^X(\mu)$~\cite{GGRuelle}. 

Since one term grows quadratically and the other one linearly on $t_n$, in the right-hand side of Equation~\eqref{Eq:Framings}, the term $\frac 1{t_n^2}R^X(\gamma_n)$ is negligible, and the asymptotic is dictated by~$ \Slk^{DX}(\gamma_n)$. 
In particular we have
\begin{equation*}
\frac 1{t_n^2} \chi_{min}(\gamma_n) =-\frac{1}{t_n^2}\Slk^{\zeta_X}(\gamma_n)\to -\Hel(X, \mu).
\end{equation*}
Then
$$\lim_{t_n\to \infty} \frac{1}{t_n^2}g(\gamma_n)=\lim_{t_n\to \infty} \frac{1}{t_n^2}\left(1+\frac{\Slk^{\zeta_X}(\gamma_n)-1}{2}\right)=\lim_{t_n\to \infty} \frac{\Slk^{\zeta_X}(\gamma_n)}{2t_n^2}=\frac{1}{2}\Hel(X, \mu).$$
\end{proof}

In other words, the genus is an asymptotic invariant of order~$2$ for right-handed volume-preserving vector fields, and its asymptotic is half the helicity. 

Remark that Baader proved that the slice genus (for arbitrary vector fields, not only right-handed) is an asymptotic invariant of order~$2$, and that it is also equal to half the helicity~\cite{Baader}. 
So for right-handed vector fields, the long periodic orbits tend to have genus and slice genus of the same order. 


\bibliographystyle{siam}

\end{document}